\documentclass{amsart}

\usepackage{amssymb,amsfonts,latexsym,amsmath,amsthm,graphicx}
\numberwithin{equation}{section}
\newtheorem{thm}{Theorem}[section]
\newtheorem{prop}[thm]{Proposition}
\newtheorem{cor}[thm]{Corollary}
\newtheorem{lemma}[thm]{Lemma}
\newtheorem{conj}[thm]{Conjecture}
\theoremstyle{remark} 
\newtheorem{remark}[]{Remark}

\newcommand{\RR}{\mathbb{R}}
\newcommand{\CC}{\mathbb{C}}

\newcommand{\DD}{\mathbb{D}}

\usepackage{xcolor}

\newcommand{\Log}{\mathrm{Log}}

\def\const{\mathrm{const}}

\newcommand{\pd}{\partial}

\newcommand{\ol}{\overline}

\newcommand{\e}{\varepsilon}

\newcommand{\p}{\partial}

\newcommand{\be}{\begin{equation}}
\newcommand{\ee}{\end{equation}}

\newcommand{\cn}{\textrm{cn}}
\newcommand{\sn}{\textrm{sn}}
\newcommand{\dn}{\textrm{dn}}

\begin{document}

\title[Overdetermined problem in potential theory]{An overdetermined problem in potential theory}
\date{May 2012}

\author[Dmitry Khavinson]{Dmitry Khavinson}
\email{dkhavins@usf.edu}
\address{Dept. of Mathematics and Statistics, University of South Florida, Tampa, FL 33620}

\author[Erik Lundberg]{Erik Lundberg}
\email{elundber@math.purdue.edu}
\address{Dept. of Mathematics, Purdue University, West Lafayette, IN 47906}

\author[Razvan Teodorescu]{Razvan Teodorescu}
\email{razvan@usf.edu}
\address{Dept. of Mathematics and Statistics, University of South Florida, Tampa, FL 33620}


%\thanks{The first author was supported by...}

\begin{abstract}
We investigate a problem posed by L. Hauswirth, F. H\'elein, and F. Pacard \cite{HHP2011}, namely, 
to characterize all the domains in the plane that admit a ``roof function'', i.e., 
a positive harmonic function which solves simultaneously a Dirichlet problem with null boundary data, 
and a Neumann problem with constant boundary data. 
As they suggested, we show, under some a priori assumptions,
that there are only three exceptional domains: the exterior of a disk, a halfplane, 
and a nontrivial example found in \cite{HHP2011} that is the image of 
the strip $|\Im \zeta| < \pi/2$ under $\zeta \rightarrow \zeta + \sinh(\zeta)$.
We show that in $\RR^4$ this example does not have any axially symmetric analog
containing its own axis of symmetry.

%Besides solving 
%the original problem posed in \cite{HHP2011} for the two-dimensional case, we discuss how this problem is related to other interesting 
%topics, such as the (multi-cut) radial Loewner equation, the constrained energy-minimization problem on the circle, and representation theory 
%of orientation-preserving circle diffeomorphisms, modulo global rotations (a subspace of  $Vect(S^1)$). 
\end{abstract}

\maketitle

\section{Introduction}

In \cite{HHP2011}, the authors have posed the following problem: 
find a smooth bounded domain $\Omega$ in a Riemannian manifold $\mathcal{M}_g$ with metric $g$, 
such that the first eigenvalue $\lambda_1$ of the Laplace-Beltrami operator on $\Omega$ has a corresponding real, 
positive eigenfunction $u_1$ satisfying $u_1 = 0, \frac{\p u_1}{\p n} = 1$ on the boundary of $\Omega$. 
Any such domain is called \emph{extremal} because it provides a local minimum for the first eigenvalue $\lambda_1$ of the Laplace-Beltrami operator, 
under the constraint of fixed total volume of $\Omega$ (see \cite{HHP2011} and references therein). 

In special cases one can find a sequence of extremal domains $\{ \Omega_t \}$ with increasing volumes, 
such that the limit domain $\Omega = \Omega_{t \to \infty}$ is unbounded, and its first eigenvalue vanishes as $t \to \infty$. 
This limit extremal domain is then called \emph{exceptional}, and the corresponding limit function $(u_{1,t})_{t \to \infty} \to u$ is a positive, 
harmonic function on $\Omega$ which solves simultaneously the overdetermined boundary value problem with null Dirichlet data and constant Neumann data. 

The problem of finding exceptional domains in $\mathbb{R}^n$ 
and their corresponding functions $u$ (called ``roof" functions by the authors of \cite{HHP2011}) 
is a nontrivial problem of potential theory. 
There is no obvious variational principle to use, 
on the one hand because $\Omega$ is unbounded (so the Dirichlet energy of $u$ \cite[Ch. 1]{Astala} will diverge), 
and, on the other hand, because the constant Neumann data constraint is not conformally invariant. 

In the absence of a suitable variational formulation, we may interpret the scaling $t \to \infty$ described above as a dynamical process, 
in which the pair $(\Omega_t, u_t)$ evolves so that the limit $t \to \infty$ solves the overdetermined problem. 
In other words, we can turn this observation into a constructive method for finding (building) exceptional domains. 
In order to do this, it is helpful to note that, upon compactification of the boundary $\p \Omega$ (with metric $d\sigma^2$), 
the pair $(\Omega, u)$ with flat metric becomes conformal to the half-cylinder $\mathcal{N} := \mathbb{R}_+ \times \overline{\p \Omega}$, with metric 
$$
ds^2 = e^{-2u} (du^2 + d\sigma^2).
$$ 
Under this reformulation, scaling of $(\Omega_t, u_t)_{t \to \infty}$ becomes equivalent to scaling of the metric structure given above, 
defined over the fixed space $\mathcal{N}$.  
This is reminiscent of the Ricci flow, in which the metric structure $g$ evolves with respect to a deformation parameter $t \in \mathbb{R}$ according to the equation 
$$
\frac{d g_{ij}}{dt} = - 2 R_{ij}, 
$$
with the right side of the equation given by the covariant Ricci tensor. 
It is known \cite{Top} that for the case of a two-dimensional manifold, with metric given by $ds^2 = e^{-2u}(dx^2 + dy^2)$, 
the Ricci flow equations reduce to a single nonlinear equation
$$
\frac{\p u}{\p t} = \nabla^2_{g} u 
$$
(since in two dimensions the Riemann tensor has only one independent component).
This is a heat equation with the generator given by the Laplace-Beltrami operator corresponding to the metric $ds^2$.
Therefore, if there is a stationary solution $\frac{\p u}{\p t} \to 0$ as $t \to \infty$, 
it will correspond to the scaling of the first eigenvalue $\lambda_1(t) \to 0$ and, 
by conformally mapping back $\mathcal{N}$ using the solution $u (t \to \infty)$, 
we will obtain the solution $(\Omega, u)$. 

In other words, we can summarize this constructive method for finding exceptional domains in $\mathbb{R}^2$ as follows: 
starting from a 2-dimensional Riemannian manifold with finite volume and metric encoded through the positive real function $u$, 
and boundary set defined via $u = 0$, consider the time evolution given by the Ricci flow, without volume  renormalization. 
Then \cite{Top} the manifold will remain Riemannian at all times, 
and in the $t \to \infty$ limit the function $u$ will become a solution of the nonlinear Laplace-Beltrami equation. 
Furthermore, if $u$ remains finite everywhere in the domain, then it is harmonic and satisfies both Dirichlet and Neumann conditions at all finite boundary components, 
so it is a solution for the overdetermined potential problem. Considered together with the (boundary) point at infinity, 
the manifold is equivalent \cite{Poincare} to a pseudosphere (flat everywhere except at the infinity point, with overall positive curvature).  
(We wish to emphasize that there is no reason to assume that such constructive methods would be exhaustive.)

Thus, so motivated, it is natural to try to characterize exceptional domains in flat Euclidean spaces.
The authors in \cite{HHP2011} suggested that in two dimensions there are only three examples:
a complement of a disk, a halfplane, 
and a nontrivial example obtained as the image of the strip $|\Im \zeta| \le \pi/2$ under the mapping $\zeta \rightarrow  \zeta + \sinh(\zeta)$.  
They posed as an open problem to determine if these are the only examples \cite[Section 7]{HHP2011}.
(They gave some evidence by characterizing the halfplane under a global assumption on the gradient of the roof function \cite[Prop. 6.1]{HHP2011}.)
They also posed the problem of finding nontrivial examples in higher dimensions and suggested the possibility of
axially symmetric examples similar to the nontrivial example in the plane \cite[Remark 2.1]{HHP2011}.

We address both of these problems.
The paper is organized as follows.  In Section 2, we review the theory of Hardy spaces in order to address a subtlety that arises
in connection with the regularity of the boundary of an exceptional domain.
This leads us to assume in our theorems that the domain $\Omega$ is Smirnov.
In Section 3, we characterize exteriors of disks as being the only exceptional domain whose boundary is compact.
In Section 4, we establish a connection between the ''roof function'' of an exceptional domain and the so-called \emph{Schwarz function} of its boundary,
and we also show that the boundary of a simply connected exceptional domain $\Omega$ can pass either (i) once or (ii) twice through infinity.
In Section 5, we show that Case (i) implies that $\Omega$ is a halfplane.
In Section 6, we show that Case (ii) implies that $\Omega$ is the nontrivial example found in \cite[Section 2]{HHP2011}.
In each of these theorems we assume that $\Omega$ is Smirnov, but we allow the roof function to be a weak solution
merely satisfying the boundary conditions almost everywhere.

In Section 7, we extend the result of Section 3 to higher dimensions.
In Section 8, we show that the nontrivial example from Section \ref{sec:nontrivial} does not allow an extension to axially symmetric domains in four dimensions, 
contrary to what was suggested in \cite[Remark 2.1]{HHP2011} (and we conjecture that this example has no analogues in any number of dimensions greater than two).

Sections 3 through 6 together partially confirm what was suggested in \cite[Section 7]{HHP2011}
under some assumptions on the topology of $\Omega$ and assuming that $\Omega$ is Smirnov.
In Section 9, we give concluding remarks including a conjecture that, up to similarity, 
there are only three \emph{finite genus} exceptional domains.
The additional assumption of finite genus is due to a remarkable example of an infinite-genus exceptional domains that appeared in the fluid dynamics literature
\cite{BSS76}.
See Section 9 for discussion.

\remark After this paper was submitted, Martin Traizet announced a more complete classification of exceptional domains \cite{MT2013}
after developing an exciting new connection to minimal surfaces.
He characterized the three examples among those having finitely many boundary components.
M. Traizet's preprint that appeared while we have been revising the previous version of our paper
finds a new beautiful connection of the problem with the theory of minimal surfaces.
From this point of view, he noticed the above-mentioned family of infinite genus examples \cite{BSS76} 
and also characterized them among periodic domains for which the quotient by the period has finitely many boundary components \cite[Theorem 13]{MT2013}.
For this latter result, he invokes a powerful theorem of W. H. Meeks and M. Wolf.
Our methods mostly rely on classical function theory ($H^p$ spaces) and potential theory and in most parts are different from Traizet's. 
Interestingly, as Traizet notes in his preprint \cite[Remark 5]{MT2013}, if one could prove his Theorem 13 by only invoking pure function theory 
this would give (via Traizet's results) a new and independent proof of the Meeks-Wolf result from minimal surfaces.
An attractive challenge!

\noindent {\bf Acknowledgement:} The authors are indebted to Dimiter Vassilev for bringing the article \cite{HHP2011} to their attention.
We wish to thank Alexandre Eremenko for sharing an improved proof of Theorem \ref{thm:Martin}
and Arshak Petrosyan and Koushik Ramachandran for pointing out the example of a cone as an exceptional domain.
We also wish to thank Martin Traizet for helpful discussion regarding his preprint.
The two first named authors acknowledge partial support from the NSF under the grant DMS-0855597.

\section{Classical vs. Weak Solutions, Regularity of the Boundary, and Hardy Spaces}\label{sec:Smirnov}

From the rigidity of the Cauchy problem, one might expect to obtain, ``for free'', regularity of the boundary of an exceptional domain 
(as is often the case for solutions of free boundary problems).
Unfortunately, the problem at hand is complicated by a remarkable family of examples with rectifiable but non-smooth boundaries, a.k.a. 
non-Smirnov domains - cf. \cite[Ch. 10]{Duren}. This results in adding a Smirnov condition to the assumptions on the domains if we desire to consider "weak solutions", 
i.e., harmonic "roof functions" satisfying the Dirichlet and Neumann boundary conditions almost everywhere with respect to the Lebesgue measure.

In order to address this subtlety, 
we first give some background from $H^p$ theory - cf. \cite{Duren} for details.

An analytic function $f:\DD \rightarrow \CC$ is said to belong to the Hardy class $H^p$, $0 < p < \infty$, if the integrals:
$$ \int_{0}^{2\pi}{|f(re^{i \theta})|^p d\theta} $$
remain bounded as $r \rightarrow 1$.

Recall that a \emph{Blaschke product} is a function of the form
$$ B(z) = z^m \prod_{n} \frac{|a_n|}{a_n}\frac{a_n - z}{1 - \ol{a_n} z},$$
where $m$ is a nonnegative integer and $\sum(1-|a_n|) < \infty$.
The latter condition ensures convergence of the product (See Theorem 2.4 in \cite{Duren}). 

A function analytic in $\DD$ is called an \emph{inner function} if its modulus is bounded by $1$ and its modulus has radial limit $1$ almost everywhere on the boundary.
If $S(z)$ is an inner function without zeros, then $S(z)$ is called a \emph{singular inner function}.

An \emph{outer function} for the class $H^p$ is a function of the form
\begin{equation}\label{eq:outer}
 F(z) = e^{i \gamma} \exp \left\{ \frac{1}{2 \pi} \int_{0}^{2 \pi} \frac{e^{it} + z}{e^{it} - z} \log \psi(t) dt \right\},
\end{equation}
where $\gamma$ is a real number, $\psi(t) \geq 0$, $\log \psi(t) \in L^1$, and $\psi(t) \in L^p$.

The following theorem \cite[Ch. 2, Ch. 5]{Duren} (also cf. \cite{fisher}) provides the parametrization of functions in Hardy classes by their zero sets, 
associated singular measures, and moduli of their boundary values.

\begin{thm}\label{thm:factorization}
Every function $f(z)$ of class $H^p$ ($p>0$) has a unique (up to a unimodular constant factor) factorization of the form $f(z) = B(z) S(z) F(z)$,
where $B(z)$ is a Blaschke product, $S(z)$ is a singular inner function, and $F(z)$ is an outer function for the class $H^p$.
\end{thm}

Suppose $\Omega$ is a Jordan domain with rectifiable boundary and $f:\DD \rightarrow \Omega$ is a conformal map.
Then $f' \in H^1$ by Theorem 3.12 in \cite{Duren}.
By Theorem \ref{thm:factorization}, $f'$ has a canonical factorization $f'(z) = B(z) S(z) F(z)$,
and since $f$ is a conformal map $f'$ does not vanish, so $f'(z) = S(z) F(z)$.
Then $\Omega$ is called a \emph{Smirnov domain} if $S(z) \equiv 1$ so that $f'(z) = F(z)$ is purely an outer function.
This definition is independent of the choice of conformal map.

There are examples of non-Smirnov domains with, as above, 
$f'(z) = S(z) F(z)$, but now $F(z) \equiv 1$ and the singular inner function $S(z)$ is not constant.
Such examples were first constructed by M. Keldysh and M. Lavrentiev \cite{KeldLavr} using complicated geometric arguments.
Their existence was somewhat demystified by an analytic proof provided by P. Duren, H. S. Shapiro, and A. L. Shields \cite{DSS}.
Like the disk, such a domain has harmonic measure at zero (assuming $f(0) = 0$) proportional to arc-length.
Thus, its boundary is sometimes called a ``pseudocircle''.

Similarly, there are ``exterior pseudocircles'',
arising as the boundary of an unbounded non-Smirnov domain \cite{JonesSmirnov} for
which the harmonic measure at infinity is proportional to arclength,
and thus Green's function with singularity at infinity
provides a roof function that is a weak solution satisfying the boundary conditions almost everywhere.
Thus, this provides a pathological example of an exceptional domain in a weak sense. 
In order to construct such an unbounded non-Smirnov domain, let us follow the method in the above mentioned \cite{DSS},
which is presented in Duren's book \cite[Section 10.4]{Duren}.
We recall that the construction is carried out by ``working backwards'', first writing down a singular inner function $S(z)$
as a candidate for the derivative $f'(z)$ of the conformal map $f(z)$.
The difficulty is then to show that $f(z)$ is not only analytic, but is also \emph{univalent} so that it actually gives a conformal map from $\DD$ to some domain $\Omega$.
Univalence is established using a criterion of Nehari which states that the following growth condition on the Schwarzian derivative $(S f) (z)$
is sufficient for univalence:
\begin{equation}\label{eq:Nehari}
 (S f) (z) \leq \frac{2}{(1-|z|^2)^2}.
\end{equation}

Let us follow this procedure, indicating the step that needs to be modified.
Start with a measure $\mu \leq 0$, singular with respect to Lebesgue measure on the circle, yet sufficiently smooth,
so that it belongs to the Zygmund class $\Lambda_*$ (cf. \cite[Section 10.4]{Duren}).

We will also require $\mu$ to have the first moment zero, i.e.,
\begin{equation}\label{eq:mu}
 \int_0^{2 \pi} e^{i \theta} d \mu(\theta) = 0.
\end{equation}
This can always be acheived by symmetrizing $\mu$ around the origin and replacing $\mu$
by $\frac{1}{2} (d \mu (\theta) + d \mu (- \theta))$.
Then the center of mass is at the origin, which is (\ref{eq:mu}).

As in \cite{Duren}, let $F(z)$ be the Schwarz integral of $\mu$

$$ F(z) = \frac{1}{2 \pi} \int_0^{2 \pi} \frac{e^{i \theta} + z}{e^{i \theta} - z} d\mu(\theta) .$$

Let $g(z)$ be the exponential of a constant (to be chosen later) times $F$,
$$g(z) = \text{exp} \{ -a F(z) \}.$$
Here is where we depart slightly from \cite{Duren} in order to get an unbounded domain as the image of $f(z)$.
Instead of taking $g(z)$ as a candidate for $f'(z)$, we take
$$f'(z) = g(z) / z^2.$$
Note that the residue of $f'(z)$ is zero (from having made the first moment of $\mu$ zero) 
so that its antiderivative $f(z)$ is analytic in $\DD$ except for a simple pole at $z=0$.
Also, $|f'(z)| = 1$ a.e. on $\p \DD$.

A calculation shows that the Schwarzian derivative $S f$ of $f$ is:
$$ (S f ) (z) = (S g) (z) - \frac{2}{z} \frac{g'(z)}{g(z)} = - a F''(z) - \frac{a^2}{2} F'(z)^2 + \frac{2 a}{z} F'(z) .$$
As explicitly stated in \cite[Section 10.4]{Duren}, $F''(z)$, $F'(z)^2$, and $F'(z)$ are each 
$$O \left( \frac{1}{(1-|z|)^2} \right) .$$
Moreover, by the vanishing of the first moment of $\mu$, $F'(0) = 0$,
so that $F'(z) / z$ is also $O \left( \frac{1}{(1-|z|)^2} \right) $.
Thus, for a small enough choice of $a$,
$(S f) (z)$ satisfies the Nehari criterion for univalence (\ref{eq:Nehari}).

Hence, $f(z)$ is a conformal map mapping $\{ |z| < 1 \}$ onto the complement of a Jordan domain with rectifiable boundary.
To see why the boundary is rectifiabile, note that, as stated in \cite[Section 10.4]{Duren}),
$g(z) \in H^1$, and so $f'(z) = g(z) / z^2 $ is in $H^1$ in an annulus $0<r<|z|<1$.

This seemingly excessive construction of an exterior pseudocircle cannot be avoided by
simply taking an inversion of an interior pseudocircle; the result will be non-Smirnov, but it will \emph{not} be an exterior pseudocircle.
Nor can one simply take the complement. 
As P. Jones and S. Smirnov proved in \cite{JonesSmirnov}, the complement of a non-Smirnov domain is often Smirnov!
(This unexpected resolution of a long standing problem
put to rest  all  hopes to characterize the Smirnov property in terms of a boundary curve.)

\remark
The above examples of non-Smirnov exceptional domains lead to assuming $\Omega$ is Smirnov in our main theorems (but we allow $u$ to be a weak solution).

An alternative approach is to require $u$ to be a ``classical solution'' that satisfies the boundary condition everywhere
(and not just almost everywhere), then non-Smirnov domains are ruled out.
Moreover, real-analyticity of the boundary then follows automatically.
To be precise, we have the following Lemma.

\begin{lemma}\label{lemma:reg}
If $\Omega \subset \RR^2$ is exceptional and the roof function $u$ is a ``classical solution'' in $C^1(\overline{\Omega})$, then $\partial \Omega$ is locally real-analytic.
\end{lemma}

\begin{proof}
The analytic completion $f(z) = u + iv$ (possibly multivalued) maps $\Omega$ into the right halfplane, since $u$ is positive.  
The Neumann condition for $u$ implies that $|f'(z)| = 1$ on $\partial \Omega$.
Also, $u \in C^1(\overline{\Omega})$ implies that $f' \in C(\overline{\Omega})$.

Choose a point $z_0 \in \partial \Omega$, 
and let $\zeta_0 = f(z_0)$.
Let $g(\zeta) = f^{-1}(\zeta)$ denote the local inverse of $f(z)$.
Choose a neighborhood $U$ of $\zeta$ and let $F:= \overline{U \cap \{\Re(\zeta) \geq 0 \}}$.
Choose $U$ small enough so that $g \in C(F)$.

Since $|g'(\zeta)| = 1$ on $\partial \Omega$, we can also choose $U$ small enough that $g'$ does not vanish in $F$.
This implies that $h(\zeta) = \Log( g'(\zeta) )$ is analytic in the interior of $F$ and continuous in $F$.
We have $\Re \{ h(\zeta) \}$ vanishes on the imaginary axis, since $|g'(\zeta)| = 1$ there.
Thus $h(\zeta)$ extends to a neighborhood of $\zeta_0$ by the Schwarz reflection principle.
This allows us to extend $g'(z)$ and therefore $g(z)$ and $f(z)$ extend analytically across $z_0$, since $u := \Re f = 0$ 
on $\p \Omega$ and $|\nabla u| = 1$ on $\p \Omega$ near $z_0$. 
The lemma is proved. 
\end{proof}

\begin{cor}\label{cor:reg}
 If $\partial \Omega$ is $C^2$-smooth and $\Omega$ is exceptional then $\partial \Omega$ is locally real-analytic.
\end{cor}

\begin{proof}
$C^2$-smoothness of $\partial \Omega$ implies that $u$ is in $C^1(\overline{\Omega})$.
It remains now to refer to Lemma 2.2.
\end{proof}

Using Kellogg's theorem on regularity of conformal maps up to the boundary, 
cf. \cite[Ch. 3]{Pom},
one easily extends the above corollary to $C^{1,\alpha}$, $\alpha >0$, boundaries and even merely to $C^1$ boundaries. 
We shall not pursue these details here. 
It would be interesting to find sharp necessary and sufficient conditions 
for the a priori regularity of the boundary that would guarantee the conclusion of Corollary \ref{cor:reg}.
As we have mentioned in the beginning of this section, 
it is necessary to assume that the domain is Smirnov, 
but it is not at all obvious that this is indeed sufficient, cf. a related discussion in \cite{DeCastroKhav} and \cite{DeCastroKhav2} regarding nonconstant functions in $E^p$ classes with real boundary values.

%\todo{I removed the {\bf ``Ansatz''}, because we now assume Smirnov in the next Theorem.} 

\section{The Case When Infinity is an Isolated Boundary Point}\label{sec:disk}

\begin{thm}\label{disk}
Suppose $\Omega$ is an exceptional domain whose complement $\CC \setminus \Omega$ is bounded and connected,
and assume $\Omega$ is Smirnov.
Then $\Omega$ is the exterior of a disk.
\end{thm}

\begin{proof}
Let $u$ be a roof function for $\Omega$.  
Positivity of $u$ implies, by B\^ocher's Theorem \cite[Ch. 3]{ABR92}, $u(z) = u_0(z) + C \log|z|$ for some constant $C$, 
where $u_0(z)$ is harmonic in $\Omega \cup \{ \infty \}$, and $u_0(z)$ approaches a constant at infinity (the ``Robin constant'' of $\partial \Omega$).
Thus, in view of the Dirichlet data of $u$, $u(z)$ is a multiple of the Green's function of $\Omega$ with a pole at infinity,
and taking $v(z)$ to be the harmonic conjugate of $u(z) / C$, 
we have a conformal map $g(z) = e^{u(z) / C + i v(z)}$  from $\Omega$ to the exterior of the unit disk (note that $g(z)$ is single-valued in $\Omega$).

Using both the Dirichlet and Neumann data, we have $|g'(z)| = 1/C$ a.e. on $\partial \Omega$,
and therefore
$$|(g^{-1})'(\zeta)| = \frac{1}{|g'(g^{-1}(\zeta))|} = C$$ a.e. on $\partial \DD$.
Since $g^{-1}$ has a simple pole at infinity, $(g^{-1})'$ is analytic.
Also, $(g^{-1})'$ is in $H^1(\CC \setminus \DD)$ since $\partial \Omega$ is rectifiable. 
Since $\Omega$ is Smirnov, the latter function is outer and also has constant modulus on the unit circle a.e.,
which together imply that it is constant.
(Recall from Section \ref{sec:Smirnov} that by formula (\ref{eq:outer})
an outer function is determined from its boundary values.)
Hence $g^{-1}$ is a linear function and $\partial \Omega$ is a circle.
\end{proof}

We defer proving a higher-dimensional version of this result until Section \ref{sec:Reichel},
but we mention here that under more smoothness assumptions the higher-dimensional case can be proved using a theorem of W. Reichel \cite{Reichel97}.

Under additional smoothness assumptions, 
the hypothesis of Theorem \ref{disk}
guarantees that $\Omega$ is a special type of \emph{arclength quadrature domain}.
The following is then an immediate corollary of a result of B. Gustafsson
\cite[Remark 6.1]{Gustafsson87}.

\begin{thm}[B. Gustafsson, 1987]\label{thm:gust}
Suppose $\Omega$ is a finite genus exceptional domain, with piecewise-$C^1$ boundary,
and infinity is not a point on the boundary of $\Omega$.
Then $\Omega$ is the exterior of a disk.
\end{thm}

This removes the condition that the complement of $\Omega$ is connected.

\begin{proof}
We will show that $\Omega$ is an \emph{arclength null quadrature domain} for analytic functions vanishing at infinity.
At first, 
consider as a class of test functions to integrate over $\partial \Omega$ rational functions $r(z)$ in $\Omega$ vanishing at infinity.

Let $f(z) = u(z) + i v(z)$ be the analytic completion of the roof function $u$.
Note that $f'(z)$ is single-valued (since it is the conjugate of the gradient of $u$),
and by B\^ocher's theorem cited above $f'(z) = O(|z|^{-1})$.
Since the gradient of $u$ is the inward normal of $\partial \Omega$, $\ol{f'(z)} = \frac{1}{f'(z)}$ is as well.
The unit tangent vector $\frac{dz}{ds}$ is a $90$-degree rotation of the normal vector $\frac{1}{f'(z)}$.
Thus, $ i f'(z) dz = ds.$
We then have a quadrature formula for integration of $r(z)$ with respect to arclength:
\begin{equation}\label{eq:alnqd}
 \int_{\partial \Omega} r(z) ds = i \int_{\partial \Omega} r(z) f'(z) dz = 0,
\end{equation}
where the vanishing of this integral is obtained by deforming the contour to infinity
where $f'(z) r(z) = O(|z|^{-2})$.
Indeed, $r(z) = O(|z|^{-1})$ by assumption on the test class, and $f'(z) = O(|z|^{-1})$ as mentioned above.

If the boundary of $\Omega$ is piecewise-$C^1$ then the rational functions are dense in $E^p$ classes
(see \cite[Thm. 10.7]{Duren}, and for the multiply connected case, \cite{TK1}, \cite{TK2}, \cite{TK3}).
In particular, the rational functions
$r(z)$, vanishing at infinity, are dense in the space of functions $E(\Omega)$ considered in \cite{Gustafsson87}.
Thus, (\ref{eq:alnqd}) shows that $\Omega$ is an arclength null quadrature domain for this space of functions,
and the result now follows from Remark 6.1 in \cite{Gustafsson87}.
\end{proof}

\section{The Schwarz Function of an Exceptional Domain}\label{sec:Schwarz}

The \emph{Schwarz function} of a real-analytic curve $\Gamma$ is the (unique and guaranteed to exist near $\Gamma$) 
complex-analytic function that coincides with $\bar{z}$ on $\Gamma$.
For the basics on the Schwarz function we refer to \cite{Davis74} and \cite{Shapiro92}.

We recall two basic facts needed in the proof of the next proposition.

\begin{itemize}
\item[(i)] On $\Gamma$, $|S'(z)| = 1$. 
\item[(ii)] The complex conjugate of $\sqrt{-S'(z)}$ is normal to $\Gamma$.
\end{itemize}

Statement (i) follows from the chain rule and the fact that the complex conjugate of the Schwarz function, $\ol{S(z)}$, is an involution (see \cite[Ch. 7]{Davis74}).
Statement (ii) follows from the formula for the complex unit tangent vector
$T(z) = \frac{dz}{ds} = \frac{1}{\sqrt{S'(z)}}$
expressing the derivative of $z$  with respect to the arc-length along $\Gamma$ (see again \cite[Ch. 7, Formula (7.5)]{Davis74}).

\begin{prop}\label{Schwarz}
If $\Omega$ is an exceptional domain such that the roof function is a classical solution,
then the $z$-derivative of the roof function is given by $u_z(z) = c\sqrt{-S'(z)}$, where $c$ is a real constant and $S(z)$ is the Schwarz function of $\partial \Omega$.
In particular, $S'(z)$ is analytic throughout $\Omega$.
\end{prop}

\remark If, for instance, the constant Neumann data for the roof function is $1$, then the constant above $c = \pm 1/2$ where the sign depends on the orientation of the boundary.

\begin{proof}

Lemma \ref{lemma:reg} implies that $\Gamma$ is locally real-analytic.  
So $\Gamma$ has a Schwarz function $S(z)$.
The complex conjugate of the analytic function $u_z$ is normal to $\Gamma$ (since $u$ has zero Dirichlet data).
In light of the constant Neumann data, we then have $|u_z(z)| = |(u_z(z))^*| = \frac{1}{2}|u_x + iu_y| = \frac{1}{2}\sqrt{u_x^2+u_y^2}$ is constant on $\Gamma$.
This, along with the statements (i) and (ii) above, shows that on $\Gamma$ the vectors $u_z(z)$ and $\sqrt{-S'(z)}$ are parallel and each have constant length.
Therefore, for some real constant $c$, the equation $u_z(z) = c\sqrt{-S'(z)}$ holds on $\Gamma$.
But since $u_z$ and $\sqrt{-S'(z)}$ are both analytic, the equation is true everywhere that either side is defined.
In particular, this guarantees analytic continuation of $S'(z)$ throughout $\Omega$.
\end{proof}

Let us use the Schwarz function to give a heuristic argument that the boundary of an exceptional domain can pass through infinity at most twice.
In fact, the angle between consecutive arcs at infinity must be $\pi$
(and obviously there cannot be more than two such angles at infinity).
Suppose the boundary of a domain has a corner where two arcs meet at an angle different from $0$, $\pi$, or $2\pi$.
Then the derivatives of the Schwarz functions of the two arcs have a branch cut along a third arc that propagates into the domain from the corner.
To see why this is the case, note that the Schwarz function of an arc can be approximated near a point by 
the Schwarz function of the tangent line.
Thus, to first order, the jump along the branch cut is linear, so to zeroth order, the jump of $S'$ is determined by the slopes of the tangent lines.
If the angle is $0$ or $2 \pi$ then the tangent line is the same for each arc, but the orientation changes, 
so there is still a jump due to the sign change. 
In the case of an angle of $\pi$ both the tangent line and the orientation are unchanged.
Thus, for any angle other than $\pi$, $S'(z)$ has a jump across a branch cut between the two boundary boundary components.
For an exceptional domain, $u$ is a global solution throughout $\Omega$, and so Proposition \ref{Schwarz} indicates that
the Schwarz function cannot have such branch cuts.  Thus, the angle between consecutive boundary arcs at infinity can only be $\pi$,
and there can be at most two such angles.

In the above informal argument, we have assumed that each arc is real analytic at infinity, so that the Schwarz function has an expansion.
Alexandre Eremenko related to us the following proof \cite{AE} using ideas from \cite{BE} that extend techniques due to Ch. Pommerenke.
No regularity assumptions on $\p \Omega$ are required.
Also, an important part of the theorem readily extends to higher dimensions.

\vspace{.1in}

We recall that a Martin function is a positive harmonic function
$M$ in a region $\Omega$
with the property that for any positive harmonic function
$v$ in $\Omega$, the condition $v\leq M$
implies that $v=cM$, where $c>0$ is a constant.
(Often, Martin functions are called minimal harmonic functions - cf. \cite{Heins1950}.)
Martin functions on finitely connected domains are simply Poisson kernels evaluated at points of the Martin boundary, 
the boundary under Caratheodory compactification (prime ends) of the domain (see \cite{Brelot71}).

\begin{thm}[A. Eremenko \cite{AE}]
\label{thm:Martin}
 The roof function $u$ of any exceptional domain $\Omega$ is a convex combination of at most two Martin functions of $\Omega$ at infinity.
 Moreover, $u(z) = O(|z|)$, and in two dimensions we also have $\nabla u(z) = O(1)$ in $\Omega$.
\end{thm}

\remark
M. Traizet \cite{MT2013} obtained the estimate $|\nabla u| \leq 1$ in $\Omega$ for domains with finitely many boundary components
using the Phragmen-Lindelof principle.
For Smirnov domains $\Omega$ it suffices to show that $u_z$ belongs to the class $N^+$ (cf. \cite{DeCastroKhav})
in order to conclude that the analytic function $u_z$ is bounded by $1$ in $\Omega$.
However, even this assumption is not needed here, and it is possible to establish the estimate on $\nabla u$ in full generality.
Alexandre Eremenko has kindly permitted us to include his argument here.

\begin{proof}
First we note that, as observed in \cite[Lemma 1]{BE}, if $u$ is a positive
harmonic function in a disk (or a ball in higher dimensions), $D(a,R)$, of radius $R$ centered at $a$, and $u(z_1)=0$ for some
boundary point $z_1$, then
\begin{equation}\label{f}
u(a)\leq 2R|\nabla u(z_1)|.
\end{equation}
This immediately follows from Harnack's inequality for $D(a,R)$ as for $z \in D(a,R)$
$$\frac{u(a)}{R+|z-a|}\leq\frac{u(z)}{1-|z-a|}=\frac{u(z)-u(z_1)}{1-|z-a|},$$
and letting $z \to z_1$ establishes (\ref{f}).

Applying (\ref{f}) when $a \in \Omega$ and $R$ is
the distance from $a$ to $\partial \Omega$, gives $u(a)\leq 2R\leq 2(|a|+\const)$.
So $u(z)=O(|z|),$ as $z\to\infty$.

The fact that $u$ is a combination of at most two Martin functions now follows from a standard argument using Carleman's inequality, see for example \cite{K}.
For the higher-dimensional case, one must use \cite{HF} instead of \cite{K}.

Next we show, in the two dimensional case, the additional claim that
$\nabla u(z) = O(1)$.
Let $R>0$ and consider an auxilliary function
$$w_R=\frac{|\nabla u|}{u+R},$$ where $R>0$ is a parameter.
A direct computation shows that
\begin{equation}\label{dwa}
\Delta\log w_R\geq w_R^2,
\end{equation} 
and $w_R(z)=1/R$ for $z\in\partial D$. 
We claim that 
\begin{equation}\label{tri}
w_R(z)\leq 2/R,\quad z\in D,
\end{equation}
from which the result follows by letting $R \rightarrow \infty$ which gives $|\nabla u| \leq 2$ in $\Omega$.

Suppose, contrary to (\ref{tri}), that $w_R(z_0)>2/R,$ for some $z_0\in \Omega$.
Let 
$$v(z)=\frac{2R}{R^2-|z-z_0|^2},\quad z\in D(z_0,R)=\{ z:|z-z_0|< R\}.$$
Obviously, $v(z)\geq 2/R$.  A computation reveals that
$\Delta\log v=v^2.$ 
Let
$$K=\{ z\in \Omega \cap D(z_0,R): w_R(z)>v(z)\}.$$
We have $z_0\in K$, since $v(z_0)=2/R$.
Let $K_0$ be the component of $K$, containing $z_0$. Then we have
$w_R(z)=v(z)$ on $\partial K_0$, since $w_R(z)< v(z)$ on $\partial \Omega \cap D(z_0,R)$
while $v(z)=+\infty$ on $\partial D(z_0,R)$.
On the other hand,
$$\Delta(\log w_R-\log v)\geq w_R^2-v^2>0 \quad\mbox{in}\quad K_0.$$
So the subharmonic function $\log u-\log v$ is positive in $K_0$ and vanishes on the boundary,
a contradiction.
\end{proof}

\remark
This a priori estimate implies the following corollary showing that the boundaries of exceptional domains
are extremely regular.
Namely, they are locally real analytic and even parameterized from the unit circle by the antiderivative of a rational function. 
In particular, it validates the preceding argument using the Schwarz function, 
and establishes that the boundary passes at most twice through infinity each time with an angle of $\pi$.
The only additional assumptions needed here are that the domain is Smirnov (cf. Section 2) and simply connected.

\begin{cor}
\label{cor:atoms}
 Let $\Omega$ be a simply connected Smirnov domain, 
 and let $h(\zeta)$ be the conformal map from $\DD$ to $\Omega$.
 If $\Omega$ is exceptional then $h'(\zeta)$ is a rational function, and either:
 
 Case (i). $h'$ has one pole on $\partial \DD$, or 
 
 Case (ii). $h'$ has two poles on $\p \DD$.
 \end{cor}
 
 \begin{proof}
Let $u$ be a roof function for $\Omega$, and $f(z) = u + i v$ its analytic completion.
Since $u > 0$, $f(z)$ takes $\Omega$ into the right halfplane, and $f(h(\zeta))$
takes the unit disk $\DD$ into the right halfplane.
Adding an imaginary constant if necessary, we may assume that $f(h(0)) > 0$ is real. 
Then, by the Herglotz Theorem (see \cite[Ch. 3]{Hoff}, \cite[Ch. 1]{Duren}), we can represent $f(h(\zeta))$ as
\begin{equation}\label{eq:SI}
f(h(\zeta)) = \int_{\mathbb{T}} \frac{e^{i\theta}+z}{e^{i\theta}-z} d\mu(\theta),
\end{equation}
with $\mu$ positive.

Now since $\Re f(h(\zeta))$ is the pull back to $\DD$ of the roof function $u$,
which by Theorem \ref{thm:Martin} is a convex combination of at most two Martin functions,
$\mu$ consists of at most two atoms.

Thus, differentiating (\ref{eq:SI}): 
\begin{equation}\label{eq:rat}
f'(h(\zeta)) \cdot h'(\zeta) = R(\zeta),
\end{equation}
where $R(\zeta)$ is a rational function with either one or two double poles on $\p \DD$ (at the atoms of $\mu$).
Since $f'(h(\zeta))$ is a bounded analytic function in $\DD$ with $|f'(h(\zeta))|=1$ a.e. on $\p \Omega$, $f'(h(\zeta))$ is an inner function.
Moreover, $h'(\zeta)$ is an outer function, since $\Omega$ is Smirnov.

For a rational function such as $R(\zeta)$ the canonical factorization given by Theorem \ref{thm:factorization} reduces to:
$$R(\zeta) = B(\zeta) \cdot F(\zeta),$$ 
with $B$ a Blaschke product and $F$ a (rational) outer function.
(The singular factor $S(\zeta)$ is trivial, since $R(\zeta)$ has no essential singularities.)
By the uniqueness of the canonical factorization, $h'(\zeta)$ and $f'(h(\zeta))$ equal $F(\zeta)$ and $B(\zeta)$ respectively (up to multiplication by a unimodular constant).
Hence, $h'(\zeta) = F(\zeta)$ is rational, and $f'(h(\zeta)) = B(z)$ is a Blaschke product.
 \end{proof}

\section{The Case When Infinity is a Single Point on the Boundary}

In this Section, we characterize the halfplane as the only simply connected exceptional domain having infinity as a single point on the boundary.
This extends \cite[Prop. 6.1]{HHP2011} by removing the additional hypothesis $\p_x u > 0$ in \cite{HHP2011} on the roof function in $\Omega$.

\begin{thm}\label{halfplane}
Suppose $\Omega$ satisfies the hypothesis of Corollary \ref{cor:atoms}.
Then Case (i) implies that $\Omega$ is a halfplane.
\end{thm}

\remark (i) We have not been able to drop the assumption that $\Omega$ is simply connected, but as mentioned in the introduction and Section 9,
M. Traizet has recently established the result under the assumption of finitely many boundary components. \cite{MT2013}.
We note that, in the simply connected case, this assumption is stronger than ours (since we allowed for infinitely many boundary components in Section 4).

(ii) We have not been able to prove a higher dimensional version of Theorem \ref{halfplane} (cf. Section 7).

\begin{proof}
Using the same notation $f$ and $h$ from the proof of Corollary \ref{cor:atoms},
$$f\left(h(\zeta)\right) = \int_{\partial \DD} \frac{e^{i\theta}+\zeta}{e^{i\theta}-\zeta} d\mu(\theta)$$
for some finite positive measure $\mu$ on $\partial \DD$.
By assumption, we are in the case when $h'$ has one pole, and according to the proof of Corollary \ref{cor:atoms} $\mu$ is an atomic measure with a single point mass.
Without loss of generality, we can place it at the point $e^{i\theta}=1$.

Thus,
\begin{equation}\label{eq:disk}
 f\left(h(\zeta)\right) = C \frac{1+\zeta}{1-\zeta}.
\end{equation}

Differentiating (\ref{eq:disk}), 
\begin{equation}\label{eq:fp}
 f'(h(\zeta)) h'(\zeta) = 2C \frac{ 1}{(1-\zeta)^2},
\end{equation}
and as asserted in the proof of Corollary \ref{cor:atoms},
$f'(h(\zeta))$ is the Blaschke factor of the right hand side, which has no zeros,
so $f'(h(\zeta))$ is a unimodular constant.
Therefore, $f$ is a linear function and $\Omega$ is a halfplane.
\end{proof}

\section{The Case When Infinity is a Double Point of the Boundary}\label{sec:nontrivial}

In this section we characterize the nontrivial example found in \cite{HHP2011}.
Suppose $\Omega$ is a simply connected domain and $\Omega$ is exceptional.
By Corollary \ref{cor:atoms}, recall that the derivative $h'(\zeta)$ of the conformal map from the disk onto $\Omega$ is a rational function with either one or two double poles on $\p \DD$.

\begin{thm}\label{cosh}
Suppose $\Omega$ satisfies the hypothesis of Corollary \ref{cor:atoms}.
Then Case (ii) implies that $\Omega$ is, up to similarity, 
the image of the strip $|\Im w| \le \pi/2$ under the conformal map $g(w) =  w + \sinh(w)$, 
while the analytic completion of the function $u(g(w))$ is the function $f(g(w)) = \cosh(w)$.
\end{thm}

\remark The exceptional domain $\Omega$ that is the image of the strip under the conformal map 
$\zeta \rightarrow \zeta + \sinh \zeta$  is precisely the exceptional domain found by the authors in \cite{HHP2011}.
Theorem \ref{cosh} together with Theorems \ref{disk} and \ref{halfplane} and Theorem \ref{thm:Martin} show that this $\Omega$ is, 
under an assumption on the topology essentially the only nontrivial example of an exceptional domain in $\RR^2$.
It turns out that some assumption on topology is necessary as there is yet a whole one-parameter family of non-similar
exceptional domains that have infinite genus (See Section 9).  However, under the assumption of finitely many boundary components, 
the example described in Theorem \ref{cosh} is the only nontrivial example as the previously mentioned recent work of M. Traizet shows \cite{MT2013}.

\begin{proof}

Using the same notation as in the proofs of Corollary \ref{cor:atoms} and Theorem \ref{halfplane},
we have that $h'(\zeta)$ is a rational function, and according to (\ref{eq:rat}) $f'(h(\zeta))$ is as well.
This justifies applying the argument principle to study $f(h(\zeta))$ and $f'(h(\zeta))$.
Namely, we will prove the following.

{\bf Claim:} The function $f$ solves a differential equation:
\begin{equation} \label{equation}
f' = \sqrt{\frac{f - 1}{f + 1}},\quad z \in \Omega,
\end{equation}
after simple normalizations described below.

Before proving this Claim we solve the differential equation to see that it gives the desired result.
Separating variables,
$$
\int \sqrt{\frac{f + 1}{f - 1}} df  = z + C.
$$
Make the substitution $f = \cosh(w), z = w + \sinh(w)$ (fixing the constant of integration $C=0$). Now using the conditions
$$
\Re f(z(w)) = 0 \, \, \text{for}\,\, z \in \p \Omega, \,\, \text{and }\,\,   \Re f(z(w)) > 0 \,\,\rm{for } \,\,\zeta \in \Omega,
$$
and the identity $\Re  \cosh(x + iy) = \cosh(x) \cos(y)$, we find that the pre-image of the domain in the $w$-plane is the strip $|\Im w| \le \pi/2$. 
Therefore, $\Omega$ can be described as the image of the strip under the map $z(w) = w + \sinh(w)$.

\end{proof}

\begin{proof}[Proof of Claim]
We will use the argument principle to show that both sides of the Equation (\ref{equation}) provide a conformal map from $\Omega$ to $\DD$.

Starting from the formula which relates the tangent vector $T(z)$ on $\p \Omega$ and the derivative of the analytic completion $f$ of $u(z)$, 
$$
T(z) = \frac{dz}{ds} = \frac{-i} {f'(z)}= \frac{1}{\sqrt{S'(z)}}, 
$$
we obtain from the continuity of $T(z)$ through the double point at infinity (see Figure 1), that
$$
\oint_{\p \Omega} d \log f'(z) =  2 \pi i.
$$
We conclude that $f'$ is a single-sheeted covering of the unit disk by the domain $\Omega$, and that it has only one zero, 
at some point $z_0 \in \Omega$.

We may assume that $f(z_0) = 1$.
If not, say $f(z_0) = a + i b$, $a > 0$, 
then one may subtract the constant $ib$ from $f$ (this just amounts to choosing a different harmonic conjugate for the same roof function),
so we have $f(z_0) = a$.  
Then one may simply replace $1$ with $a$ in the claim, and integrating the differential equation is done similarly resulting in a dilation of the original solution.

Consider now the function defined on $\Omega$, taking values in the unit disk $\mathbb{D}$
$$
g(z) := \sqrt{\frac{f(z) - 1}{f(z) +1 }}. 
$$
Then $g$ is also a univalent map from $\Omega$ into $\mathbb{D}$.
Indeed, by the argument principle, $\frac{f(z) - 1}{f(z) +1 }$ is a branched, two-sheeted covering
of the disk, since it maps each of the two boundary components shown in Figure~1 onto $\mathbb{T}$,
Moreover, the single branch point $z_0$ is mapped to the origin,
so that taking the square root gives a single-valued analytic function.

Also, $f'(z_0) = g(z_0) = 0$.  This uniquely determines the conformal map up to a unimodular constant,
which we may assume is $1$ (after a rotation), and we then arrive at the differential equation (\ref{equation}).

\end{proof}

\begin{figure}
\begin{center}
\includegraphics[width=3in]{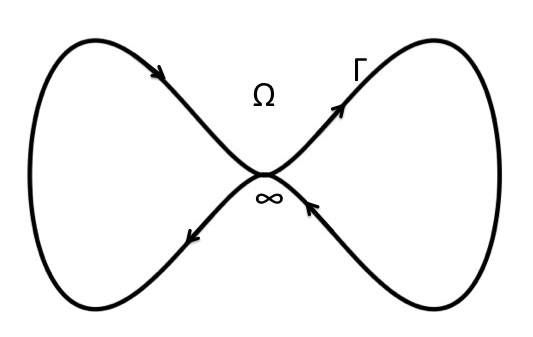}
\caption{Local geometry of the boundary $\Gamma = \p \Omega$ near infinity.}
\label{finger}
%\vspace{-0.25cm}
\end{center}
\end{figure}

\section{An Extension of Theorem \ref{disk} to Higher Dimensions}\label{sec:Reichel}

%\todo{Even though the referee didn't make any comments, I shortened the proof in this Section, making it a more direct result of Reichel's main theorem.}

In this section, we notice that some results in Section \ref{sec:disk} extend to higher dimensions.

\begin{thm}\label{thm:Reichel}
Suppose $\Omega$ is an exceptional domain in $\mathbb{R}^n$ whose exterior is bounded and connected.
If $\partial \Omega$ is $C^{2,\alpha}$-smooth, $\alpha > 0$, then $\partial \Omega$ is a sphere.
\end{thm}

\begin{proof}
 
Let $u$ be a roof function for $\Omega$,
and let $v(s)=\frac{1}{|s|^{n-2}}$ denote the Newtonian kernel.
Fix $y \in \Omega$ and take a small ball $B_{\varepsilon}$ centered at $y$.
Take also a large ball $B_R$ of radius $R$ that containes both $B_{\varepsilon}$ and the complement of $\Omega$.

Since $u(x)$ and $v(x-y)$ are harmonic in $\Omega \setminus B_\e$, Green's second identity gives
\begin{equation}
 \int_{\partial B_R + \partial \Omega - \partial B_{\varepsilon}} \left( v(x-y) \partial_n u(x) - u(x)\partial_n v(x-y) \right) d \sigma_x = 0.
\end{equation}

Letting $R \rightarrow \infty$, we can drop the integration over $\partial B_R$, 
since again by B\^ocher's Theorem \cite[Ch. 3]{ABR92}, near infinity $u(x) \approx |x|^{2-n}$.

Since, $u(x) = 0$ on $\partial \Omega$ and $\partial_n u(x) = 1$ on $\partial \Omega$,
\begin{equation}
 \int_{\partial \Omega}  v(x-y) d \sigma_x = \int_{\partial B_{\varepsilon}} \left( v(x-y) \partial_n u(x) - u(x)\partial_n v(x-y) \right) d \sigma_x.
\end{equation}

Let $U$ be the bounded domain such that $\mathbb{R}^n \setminus \overline{U} = \Omega$.
The outward normal for $\partial U$ is opposite to that of $\partial \Omega$, and since $v(x-y) = \frac{1}{\e^{n-2}}$ on $\partial B_{\varepsilon}$,

\begin{equation}
 \int_{\partial U} v(x-y) d \sigma_x = \int_{\partial B_{\varepsilon}} \left( - \tfrac{1}{\e^{n-2}} \partial_n u(x)  + u(x)\partial_n v(x-y) \right) d \sigma_x.
\end{equation}

For the first term on the right-hand-side, 
we have 
$$\int_{\partial B_{\varepsilon}} \tfrac{1}{\e^{n-2}} \pd_{n} u(x) d \sigma_x = \int_{B_{\e}} \Delta u(x) dV = 0.$$
So, 
$$\int_{\partial U} v(x-y) d \sigma_x = \int_{\partial B_{\varepsilon}} u(x)\partial_n v(x-y) d \sigma_x \rightarrow u(y),$$
as $\varepsilon \rightarrow 0$.
So, $u(y)$ is the single layer potential with charge density one on the surface $\partial U$.
That $U$ is a ball now follows from a theorem of W. Reichel \cite[Theorem 1]{Reichel97}.

\end{proof}

\remark Reichel's result holds for more general elliptic operators than the Laplacian.
In the setting of the Laplacian, J. L. Lewis and A. Vogel \cite{LV} characterized the sphere
in terms of its \emph{interior} Greeen's function under weaker regularity assumptions, namely, the boundary is assumed Lipschitz.
In that case, the Neumann condition can be assumed to hold almost everywhere on the boundary.
Thus, the hypothesis of Theorem \ref{thm:Reichel} could be weakened by checking that the same proof \cite{LV} works for the exterior case we are interested in.
Yet, we have chosen an easier and more transpoarent path to apply Reichel's result directly, 
even though it requires a stronger regularity on the boundary.

%%%%%%%%%%%%%%%%%%%%%%%%%%%%%%%

\section{Nonexistence of a Higher-dimensional Analog of the $\cosh(z)$ Example}

The authors in \cite{HHP2011} expressed a suspicion (see Remark 2.1 in \cite{HHP2011}) that there exist $n$-dimensional, 
rotationally-symmetric examples similar to the two-dimensional example $\{(x,y) \in \mathbb{R}^2: |y| < \frac{\pi}{2}+\cosh(x) \}$ 
that appeared in Section \ref{sec:nontrivial}.
We show that there does not exist an exceptional domain in $\mathbb{R}^4$ whose boundary is generated by 
rotation about the $x$-axis of the (two-dimensional) graph of an even function.

\begin{thm}\label{thm:nonexist}
 There does not exist a rotationally-symmetric exceptional domain $\Omega$ in $\mathbb{R}^4$ that contains its own axis of symmetry and 
whose boundary is obtained by rotating the (two-dimensional) graph of an even real-analytic function about the $x$-axis.
\end{thm}

\remark (i) Our proof will rely heavily on two tricks, one exploiting the assumption that $n=4$, 
and the other using the assumption that the generating curve is symmetric. 
However, we strongly suspect a more general non-existence of such examples in $\mathbb{R}^n$ for any $n>2$.
Therefore, we conjecture the following.
\begin{conj}
 For $n>2$, there does not exist an axially symmetric, exceptional domain in $\mathbb{R}^n$ that contains its own axis of symmetry.
\end{conj}
(ii) The assumption that the domain contains its axis of symmetry rules out the exteriors of balls and circular (or spherical) cylinders, 
respectively (which are clearly exceptional domains as was noted in \cite{HHP2011}).
Also, A. Petrosyan and K. Ramachandran pointed out to us that the nonconvex component of the exterior of a certain cone is also an exceptional domain.
In $\RR^4$, using the $x$-axis as the axis of rotation, the cone is the rotation of $ \{(x,y) : y^2 - x^2 = 0 \} $,
and the roof function in the meridian coordinates 
$x, y$ where $y$ is the distance to the $x$-axis in $\RR^4$, 
is $u(x,y) = \frac{y^2 - x^2}{y}$ for $y>0$.

\begin{proof}[Proof of Theorem \ref{thm:nonexist}]
Suppose that $\Omega$ is such a domain in $\mathbb{R}^4$.
Namely, the boundary $\partial \Omega$ is obtained from rotation of $\gamma:=\{ (x,y) \in \mathbb{R}^2 : y = g(x)\}$, with $g(-x) = g(x)$.
i.e., the boundary of $\Omega$ is given by 
$$\{(x_1,x_2,x_3,x_4) \in \mathbb{R}^4 : \sqrt{x_2^2+x_3^2+x_4^2} = g(x_1) \}.$$
Considering the boundary data, the rotational symmetry of the domain will be passed to the roof function, so that, abusing notation, we can write 
$$u(x_1,x_2,x_3,x_4) = u(x,y).$$
For clarity, we emphasize that the $x$-axis corresponds to the axis of symmetry and the $y$-coordinate gives the distance from the axis of symmetry.

For axially symmetric potentials $v$ in $\mathbb{R}^n$ the cylindrical reduction of Laplace's equation is:
$$\Delta_{(x,y)} v + \frac{(n-2)v_y}{y} = 0,$$
where $x=x_1$ and $y = \sqrt{x_2^2+...+x_n^2}$.
Moreover, in the case we are considering, when $n=4$, $u$ satisfies the equation $\Delta u + \frac{2 u_y}{y} = 0$,
if and only if $y u(x,y)$ is a harmonic function of two variables $x$ and $y$.
Indeed, 
$$\Delta (y u) = y \Delta u + 2 \nabla u \cdot \nabla y + u \Delta y = y \Delta u + 2 u_y.$$
(The trick that reduces axially symmetric potentials in $\mathbb{R}^4$ to harmonic functions in the meridian plane is well known, cf. \cite{Khav91} and \cite{Karp92}.)

Since $y u(x,y)$ is then harmonic in the unbounded two-dimensional domain $D$ bounded by $\gamma$ and its reflection (which we denote by $\bar{\gamma}$) with respect to the $x$-axis,
this implies $\frac{\partial}{\partial z} \left(y u(x,y) \right)$ is analytic in the domain $D$, where as usual $z = x+iy$.
The Cauchy data (originally posed in $\mathbb{R}^4$) imply that $u_z = \frac{1}{2}(u_x-iu_y)$ 
coincides with $\sqrt{-S'(z)}$ on $\gamma$ and $\bar{\gamma}$.
This implies that the analytic function 
\begin{equation}\label{eqn:analytic}
 W(z) := \left(y u \right)_z = \frac{-i}{2}u + y u_z
\end{equation}
coincides with $\frac{z-S(z)}{2i}\sqrt{-S'(z)}$ on $\gamma$ and $\bar{\gamma}$.
Since $\frac{z-S(z)}{2i}\sqrt{-S'(z)}$ is analytic, 
this actually gives a formula for $W(z)$ valid not only on $\gamma$ and $\bar{\gamma}$:
\begin{equation}\label{eqn:Schwarzformula}
 W(z) = \frac{z-S(z)}{2i}\sqrt{-S'(z)}.
\end{equation}
We note that (\ref{eqn:Schwarzformula}) can be used to analytically continue $S(z)$ to all of $D$, but this is not needed in our proof.

Let $f(\zeta)$ be the conformal map from the strip $\Sigma:= \{|\Im \zeta|<\frac{1}{2} \}$ to $D$ such that $f(0) = 0$ and $\arg \{f'(0)\} = 0$.
The two-fold symmetry of $D$ implies that $f(\zeta)$ is an odd function.
Indeed, otherwise $h(\zeta) = -f(-\zeta)$ gives another conformal map from the strip $\Sigma$ to $D$.
But, $h(0) = -f(0) = 0$ and $h'(0) = f'(0)$ implies $h = f$, by the uniqueness of the conformal map (up to choice of $f(0)$ and argument of $f'(0)$).

The Schwarz functions $S_t$, $S_b$ of the top and bottom edges of the strip $\Sigma$ are $S_t(\zeta) = \zeta - i$, and $S_b(\zeta) = \zeta + i$.  
In terms of the conformal map $f(\zeta)$, the pull-back to the $\zeta$-plane of the Schwarz functions $S_{+}$ and $S_-$ of $\gamma$ and $\bar{\gamma}$ (resp.) satisfy (see \cite[Ch. 8, Eq. $8.7$]{Davis74})

\begin{equation}
 S_{\pm}(f(\zeta)) = f(\zeta \mp i), \text{ and}
\end{equation} 
\begin{equation}
 S_{\pm}'(f(\zeta)) = \frac{f'(\zeta \mp i)}{f'(\zeta)}.
\end{equation}

Substituting these into (\ref{eqn:Schwarzformula}),
we obtain two expressions for the pullback of $W(z)$ to the strip $\Sigma$:

\begin{equation}\label{eqn:match}
\frac{f(\zeta)-f(\zeta \mp i)}{2i}\sqrt{-\frac{f'(\zeta \mp i)}{f'(\zeta)}}
\end{equation}

Even though $W(f(\zeta))$ is analytic throughout $\Sigma$, we caution that
these two expressions (one expression for ``$+$'' and one for ``$-$'') may only be valid near the bottom and top sides (respectively) of the strip $\Sigma$.

\smallskip

\noindent {\bf Claim:} The function $W(f(\zeta))$ is odd.

\smallskip

Before proving the Claim, let us see how it is used to finish the proof of the Theorem.
The fact that $W(f(\zeta))$ is odd implies $W(0) = W(f(0)) = 0$.
By (\ref{eqn:analytic}) we then have that $\frac{-i}{2}u + y u_z$ vanishes at $z=0$, which implies that $u(0,0)=0$. 
This contradicts the positivity of $u$.

\begin{proof}[Proof of Claim]
 We wish to show that $V(\zeta) = W(f(\zeta)) + W(f(-\zeta))$ vanishes identically.
Use each of the expressions in (\ref{eqn:match}) above to represent $W(f(\zeta))$ and $W(f(-\zeta))$, respectively.
\begin{equation}\label{eqn:odd}
V(\zeta) = \frac{f(\zeta)-f(\zeta - i)}{2i}\sqrt{-\frac{f'(\zeta - i)}{f'(\zeta)}} + \frac{f(-\zeta)-f(-\zeta + i)}{2i}\sqrt{-\frac{f'(-\zeta + i)}{f'(-\zeta)}}
\end{equation}

We show that this formula vanishes where it is valid, which then implies that $V(\zeta)$ vanishes identically throughout $\Sigma$.
For this, we use the fact that $f$ is odd and consequently $f'$ is even.
\begin{equation}\label{eqn:odd2}
V(\zeta) = \frac{f(\zeta)-f(\zeta - i)}{2i}\sqrt{-\frac{f'(\zeta - i)}{f'(\zeta)}} + \frac{-f(\zeta)+f(\zeta - i)}{2i}\sqrt{-\frac{f'(\zeta - i)}{f'(\zeta)}} = 0.
\end{equation}
This establishes the Claim.
\end{proof}
\end{proof}

%%%%%%%%%%%%%%%%%%%%%%%%%%%%%%%

%%%%%%%%%%%%%%%%%%%%%%%%%%%%%%%

\section{Concluding Remarks and Main Conjecture}

{\bf 1.} It is tempting to conjecture that the three examples in the plane studied above are the only exceptional domains in the plane as suggested in \cite{HHP2011}.

However, there is a remarkable family of infinitely-connected exceptional domains.
These were discovered as solutions to a fluid dynamics problem by Baker, Saffman, and Sheffield in 1976 \cite{BSS76}.
See also \cite{Crowdy} for a more detailed account.
The original problem there was to find hollow vortex equilibria with an infinite periodic array of vortices, i.e., ``spinning bubbles'' amid a stationary flow of ideal fluid.
The domain occupied by fluid turns out to be an exceptional domain with an infinite periodic array of holes,
and the roof function is a \emph{stream function} of the fluid flow, see Figure \ref{fig:periodic}.
The constant Dirichlet condition corresponds to the requirement that the boundary of each hollow vortex is a stream line,
and the constant Neumann condition corresponds to the requirement that the fluid pressure should be balanced at the interface by the pressure inside each bubble
which is assumed constant.  The latter correspondence is more subtle; 
in order to have constant pressure along a stream line, the fluid velocity (which equals the normal derivative of stream function) should be constant
according to Bernoulli's law.

\begin{figure}
\begin{center}
\includegraphics[width=3in]{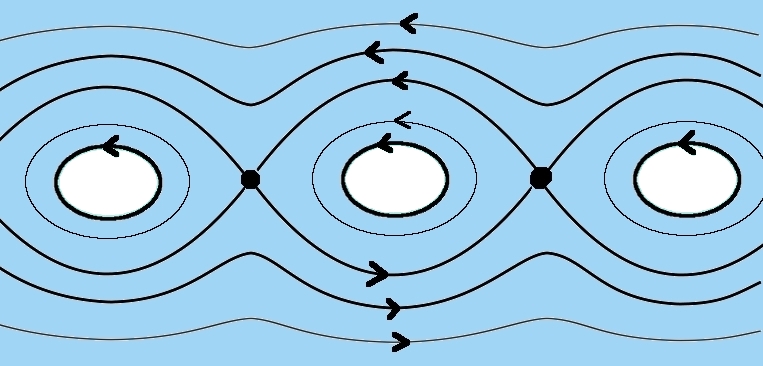}
\caption{An infinite-genus exceptional domain that also provides a hollow vortex equilibrium.  Level curves of the roof function are stream lines.  
The shape of the bubbles ensures that the pressure dictated by Bernoulli's law is constant at the fluid-bubble interface.}
\label{fig:periodic}
%\vspace{-0.25cm}
\end{center}
\end{figure}

\begin{figure}
\begin{center}
\includegraphics[width=3in]{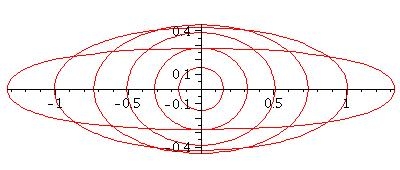}
\caption{There is actually a whole one parameter family of different bubble shapes.  
As noticed in \cite{MT2013}, each of the three previously known examples can be recovered as different scaling limits of this family.
In that sense, this family includes all known examples.}
\label{fig:periodic}
%\vspace{-0.25cm}
\end{center}
\end{figure}

This infinite genus example leads us to add to the conjecture the assumption that the domain has finite genus.

\begin{conj} \label{c9}
The only finite genus exceptional domains in $\mathbb{R}^2$ are the exterior of the unit disk, the halfplane, and the domain described in Theorem \ref{cosh}.
\end{conj}

\remark As mentioned in the introduction, Martin Traizet \cite{MT2013} recently announced a classification of exceptional domains.
His results confirm our conjecture for domains having finitely many boundary components
and also show that the above infinite genus example is the only periodic exceptional domain for which the quotient by the period has finitely many boundary components.
His methods use a remarkable nontrivial correspondence to minimal surfaces,
perturbing an exceptional domain by harmonically mapping it to another domain in such a way that the graph of the new height function
(which pulls back to the roof function in the original domain) satisfies the minimal surface equation.  
A miraculous (and crucial to his approach) by-product is that,
whereas the graph of the roof function meets its boundary at a $45$-degree angle, the minimal graph
meets its boundary vertically so that gluing it to its own reflection over the $xy$-plane results in a smooth minimal surface (without boundary!) embedded in $\RR^3$.

{\bf 2.} Regarding the higher-dimensional case, we conjecture the following extension of Theorem \ref{halfplane} to higher dimensions.
\begin{conj} \label{c9}
Suppose $\Omega$ is an exceptional domain in $\RR^n$ that 
is homeomorphic to a halfspace.  Then $\Omega$ is a halfspace.
\end{conj}

{\bf 3.} The connection to the Schwarz function in Section \ref{sec:Schwarz} reveals that exceptional domains are \emph{arclength null-quadrature domains}.
That is, for any function $f$, say analytic in $\Omega$, continous in $\ol{\Omega}$, integrable over the boundary, 
and decaying sufficiently at infinity, we have $\int_{\p \Omega} f ds = 0$.
Indeed, $\int_{\p \Omega} f ds = \int_{\p \Omega} f(z) \frac{1}{T(z)}T(z)ds = \int_{\p \Omega} f(z) \sqrt{S'(z)} dz,$
where $T(z)$ is the complex unit tangent vector (see Section \ref{sec:Schwarz}), 
and now this integral vanishes as long as the integrand decays sufficiently at infinity.
Null-quadrature were previously studied in the case of area measure.
They were characterized in the plane by M. Sakai \cite{Sakai}.
Our current study can be seen as a step toward characterizing null-quadrature domains for arclength.

{\bf 4.} Other interesting connections involve differentials on Riemann surfaces.
The study of Gustafsson \cite{Gustafsson87} used half-order differentials on the Schottky double of an arclength quadrature domain.
From a different point of view, the boundary of an exceptional domain is a trajectory of the positive quadratic differential $-(df)^2$,
where $f(z)$ is the analytic completion of the roof function.

{\bf 5.} The differential equation \eqref{equation} can be solved by a more general substitution using Jacobi elliptic functions \cite[p.~567, \S 16]{AS}: $f(\zeta, k) \equiv \cos(\theta) \cn(\zeta, k) + \sin(\theta) \sn(\zeta, k)$ and $z(\zeta) = \phi(\zeta) + \cos(\theta)\sn(\zeta, k) - \sin(\theta)\cn(\zeta, k)$, where $\phi(\zeta) = \int^{\zeta} \dn(\xi, k) d\xi$ and $\theta$ is an arbitrary phase, $\theta \in [0, 2\pi]$. 

For a given value of the elliptic modulus $k \in [0, 1]$, we define the corresponding domain $\mathbb{F}$ through its fundamental periods $T_1 (k) = 4F(\pi/2, \sqrt{1-k^2})$ and $T_2 (k) = 4 F(\pi/2, k)$, where $F(\pi/2, k) = K(k)$ is the complete elliptic integral of the first kind \cite[p. 590, \S 17.3]{AS}:
$$
K(k) \equiv \int_0^{\pi/2}\frac{1}{\sqrt{1 - k^2\sin^2(\theta)}}d \theta
$$ 
It diverges for $k=1$ and equals $\pi/2$ for $k=0$. 

Then it is straightforward to check that \eqref{equation} is satisfied by $f(z)$, due to the identity \cite[p. 573, \S 16.9]{AS}:
$$
1 = [-\sn(z) \cos(\theta) + \cn(z) \sin(\theta)]^2 + [\cn(z) \cos(\theta) +\sn(z) \sin(\theta)]^2.
$$

%\end{remark}

Let $\gamma$ be the pre-image of $\p \Omega$ under $z(\zeta)$: it consists of two pieces $\gamma_{\pm}, \gamma_{-} = - \gamma_{+}$, dividing the  fundamental domain 
$\mathbb{F}$ into three sub-domains. Denote the component which contains the origin by $D_0$, then since $f(0) = 1$, we conclude that $\Re f(z) > 0$ for $z \in D_0 \setminus \gamma_{\pm}$, 
and we have proven the following result.

\begin{prop}\label{domains}
The exceptional domain $\Omega$ is the image of the domain $D_0(k)$ under the map $z(\zeta) = \phi(\zeta) + \cos(\theta)\sn(\zeta, k) - \sin(\theta)\cn(\zeta, k)$.
\end{prop}

\begin{remark}\label{two}
The case discussed in the proof of the theorem corresponds to the degenerate elliptic modulus $k = 0$. Then the domain $\mathbb{F}$ becomes the infinite strip
$$
T_1(0) = 4 K(1) \to \infty, \quad T_2 (0) = 4 K(0) = 2\pi,
$$
while the functions $f, g$ become (using the fact that $\dn(z, 1) \equiv 1$)
$$
z(\zeta) = \zeta + \sinh(\zeta),  \quad 
f(z(\zeta)) = \cosh(\zeta).
$$
As noted before, the conditions $\Re f(\zeta)|_{\gamma_{\pm}} = 0$ give the pre-image  $\gamma_{\pm}:=  z^{-1}(\p \Omega) = \{\Im \zeta = \pm \frac{\pi}{2} \}$, 
and the pre-image of the domain, $D_0$, becomes the strip $|\Im\zeta| \le \frac{\pi}{2}$.
\end{remark}
\par
%\begin{remark}

{\bf 6.} Note that the domain $D_0(k)$ is the pre-image of the unit disk under the map $\zeta(w) : \mathbb{F} \to \mathbb{D}$,
$$
\zeta(w) = \frac{\sn(w, k) - i}{\sn(w, k) + i}, \quad k \in [0, 1],
$$
with the support of $\mu$ at points $\zeta_{\pm} = \pm \frac{1-ik}{1+ik}$, where $\mu$ is the measure discussed in the proof of Corollary \ref{cor:atoms}. The case $k \to 0$ corresponds to the strip domain and to $\zeta_{\pm} = \pm 1$. 
The reparametrization invariance noted above for the solution $f(z)$ of \eqref{equation} under rescaling of the elliptic modulus $k$ is indicative of a deeper invariance of the solution: 
all the specific solutions in $\mathbb{C}$ discussed here are associated with fixed points in the moduli space of Riemann surfaces. 

Let again $f(h(z))$ be the analytic completion of a solution, 
and denote by $\mathcal{G}$ the group of transformations which leaves ${\rm{supp}} (\mu)$ invariant up to a global rotation.
It follows that $f$ is an automorphism of the quotient of the group of linear fractional transformations by $\mathcal{G}$, 
which can be in general a Kleinian group \cite{Poincare}. 
The limit set (accumulation points of the orbits of the group) can be finite (in which case it can consist of only 0, 1, or 2 points), or infinite. 
It is known (\cite{Astala}, Thm. 10.3.4.) that the set of homeomorphic solutions for a quasilinear elliptic equation of Laplace-Beltrami type forms a group only in the case of finite limit set \cite{Astala}.  
The Kleinian groups are called degenerate in this case, and they correspond to either finite groups (with empty limit set), 
or the cyclic groups (generated by one element, with limit set consisting of 1 or 2 points). 
These correspond to the solutions described in the present paper (isolated point at infinity, respectively simple and double boundary point at infinity).

%Acknowledgments: 

\bibliographystyle{amsplain}

\end{document}